\documentclass[12pt]{amsart}

\usepackage{amsmath}
\usepackage{amsthm,amssymb,amsfonts}
\usepackage{verbatim}
\usepackage[usenames,dvipsnames]{color}
\usepackage[normalem]{ulem}

\theoremstyle{plain}
\newtheorem{thm}{Theorem}
\newtheorem{lem}[thm]{Lemma}

\newtheorem{prop}[thm]{Proposition}
\newtheorem{fact}[thm]{Fact}

\newtheorem*{thm*}{Theorem}

\theoremstyle{remark}
\newtheorem{rmk}[thm]{Remark}

\theoremstyle{definition}
\newtheorem{defn}{Definition}

\numberwithin{equation}{section}

\newenvironment{PfofThmMain}[1]
{\par\vskip2\parsep\noindent{\sc Proof of Theorem\ \ref{Thm:MainResult}.}}{{\hfill
$\Box$}
\par\vskip2\parsep}

\def\J{\mathcal{J}}

\def\Z{\mathbb Z}
\def\R{\mathbb R}

\newcommand{\mbwhy}[1]{{\textcolor{red}
{MBwhy: #1}}}

\newcommand{\jb}[1]{{\textcolor{Mahogany}
{JB: #1}}} 
\newcommand{\jbwhy}[1]{\marginpar{\tiny\textcolor{Mahogany}{[\sl
{JB: #1}\rm]}}}

\title[Bowen's entropy-conjugacy conjecture]{Bowen's entropy-conjugacy conjecture is true up to finite index}
\author{Mike Boyle}
\address{Department of Mathematics - University of Maryland}
\email{mmb@math.umd.edu}
\author{J\'er\^ome Buzzi}
\address{Laboratoire de Math\'ematiques d'Orsay - Universit\'e Paris-Sud}
\email{jerome.buzzi@math.u-psud.fr}
\author{Kevin McGoff}
\address{Department of Mathematics - Duke University}
\email{mcgoff@math.duke.edu}
\begin{document}



\begin{abstract}
For a topological dynamical system $(X,f)$, consisting of a continuous map $f : X \to X$, and a (not necessarily compact) set $Z \subset X$, Bowen \cite{Bowen1973} defined a dimension-like version of  entropy,  $h_X(f,Z)$. In the same work, he introduced a notion of entropy-conjugacy for pairs of invertible compact systems: the systems $(X,f)$ and $(Y,g)$ are \textit{entropy-conjugate} if there exist invariant Borel sets $X' \subset X$ and $Y' \subset Y$ such that $h_X(f,X\setminus X') < h_X(f,X)$, $h_Y(g,Y \setminus Y') < h_Y(g,Y)$, and $(X',f|_{X'})$ is topologically conjugate to $(Y',g|_{Y'})$. Bowen conjectured that two mixing shifts of finite type are entropy-conjugate if they have the same entropy. We prove that two  mixing shifts of finite type with equal entropy and left ideal class are entropy-conjugate. Consequently, in every entropy class Bowen's conjecture is true up to finite index.
\end{abstract}

\maketitle

\section{Introduction} \label{Sect:Introduction}

It is well-known that within many classes of dynamical systems, entropy characterizes most of the dynamics. The prime example is Ornstein's complete classification of Bernoulli schemes by their entropy up to measure-preserving isomorphisms: these are Borel conjugacies after discarding subsets of zero measure for a single distinguished invariant probability measure. It has been discovered since that in many settings, one can obtain \emph{more regular} conjugacy and/or  \emph{smaller} discarded subsets.

For instance, the notion of probabilistic entropy-conjugacy introduced in \cite{Buzzi1997} (under the name of entropy-conjugacy) neglects subsets that are negligible with respect to all invariant, ergodic probability measures with entropy close to the maximum and allows the classification of interval dynamics among others (see \cite{Buzzi2005}). Hochman \cite{Hochman} has shown that this can often be strengthened to Borel conjugacies neglecting only sets of zero measure for all non-atomic invariant, ergodic probability measures.

A classical result \cite{AdlerMarcus} of Adler and Marcus shows that mixing shifts of finite type 
are topologically conjugate after restriction to doubly transitive sequences (see \cite{BBG} for an extension to the non-compact case).  Looking at their proof, one can see that they 
produce a probabilistic entropy-conjugacy, to which the previously mentioned result of Hochman applies. The Borel conjugacies thus obtained are not homeomorphisms. Hence a natural question is, how far can the discarded set be reduced while keeping the continuity of the conjugacy?

Bowen introduced a dimension-like notion of entropy $h_X(f,Y)$ for arbitrary subsets $Y\subset X$ (see \cite{Misiurewicz,Pesin} for recent discussions of this notion) and used it to 
define another version of entropy-conjugacy, as follows. When necessary, we will 
refer to Bowen's version as  \textit{topological entropy-conjugacy}, to distinguish it 
from the probabilistic version of \cite{Buzzi1997}.

\begin{defn}\cite{Bowen1973}
Suppose $(X,f)$ and $(Y,g)$ are topological dynamical 
systems (homeomorphisms of compact metrizable spaces). 
 Then $(X,f)$ is \textbf{ (topologically) entropy-conjugate} to $(Y,g)$ if there exist Borel subsets $X' \subset X$ and $Y' \subset Y$ such that the following conditions are satisfied:
\begin{enumerate}
\item $f(X') \subset X'$ and $g(Y') \subset Y'$;
\item $h_X(f, X \setminus X') < h_X(f,X)$ and $h_Y(g, Y \setminus Y') < h_Y(g,Y)$; 
\item $(X',f|_{X'})$ is topologically conjugate to $(Y',g|_{Y'})$.
\end{enumerate}
\end{defn}

To the best of our knowledge, it is unknown whether entropy-conjugacy is transitive. It is stronger than probabilistic entropy-conjugacy, since  $h_X(f,X\setminus X')$ bounds the entropy of any ergodic, invariant measure $\mu$ carried on $X\setminus X'$  \cite[Prop. 1]{Bowen1973}.

Bowen showed that entropy-conjugate homeomorphisms of compact metrizable spaces have the same topological entropy. His paper includes a single conjecture: mixing shifts of finite type (SFTs) with the same topological entropy must be entropy-conjugate. Given the continuing influence of Bowen's paper (e.g. see \cite{Misiurewicz}) and the depth of his intuition, it seems to us that his conjecture merits attention. 

Let us also mention a related question of Hochman \cite[Problem 1.9]{Hochman}. Let $X$ and $Y$ be mixing SFTs on finite alphabets with $h(X)=h(Y)$. Let $X'$ and $Y'$ denote the sets obtained by removing all periodic points from $X$ and $Y$, respectively. Is there a topological conjugacy between the (non-compact) systems $X'$ and $Y'$?

We will show that if two mixing SFTs have the same entropy and the same left ideal class, then they are entropy-conjugate. Because there are finitely many ideal classes within a given entropy class \cite{BMT}, it follows that in every entropy class Bowen's conjecture is true up to finite index.
It remains open as to whether Bowen's original conjecture holds.

We state our main results precisely in the next two theorems.
For a mixing SFT $(X,\sigma_X)$, denote by $\J(X)$ the left ideal class of $(X,\sigma_X)$, and denote by $h(X)$ the topological entropy of $(X,\sigma_X)$ (see Section \ref{Sect:Preliminaries} for definitions).


\begin{thm} \label{Thm:MainResult}
Let $(X,\sigma_X)$ and $(Y,\sigma_Y)$ be mixing SFTs. If $h(X) = h(Y)$ and $\J(X)=\J(Y)$, then $(X,\sigma_X)$ is entropy-conjugate to $(Y,\sigma_Y)$.
\end{thm}

It is known from \cite[Theorem 5.13]{BMT} that for a given $h\geq 0$, 
\begin{equation} \label{Eqn:FiniteCard}
\#\{\J(X) : (X,\sigma_X) \text{ is a mixing SFT with } h(X) = h\} < \infty.
\end{equation} 
As a consequence of Theorem \ref{Thm:MainResult} and (\ref{Eqn:FiniteCard}), we obtain the following theorem.

\begin{thm} \label{Thm:BowenConj}
For a given $h \geq 0$, the set of mixing SFTs with entropy $h$ may be partitioned into finitely many classes $C_1, \dots, C_N$ such that if $(X,\sigma_X)$ and $(Y,\sigma_Y)$ are in $C_i$ for some $i$, then $(X,\sigma_X)$ is entropy-conjugate to $(Y,\sigma_Y)$.
\end{thm}

For some entropies (for example, entropy $\text{log}\, n$ with $n\in \mathbb N$), there is only one ideal class; in these classes, Bowen's conjecture is true.



In Section \ref{Sect:Preliminaries}, we provide some necessary definitions and background. Section \ref{Sect:Proof} is devoted to proving Theorem \ref{Thm:MainResult}.

\section{Preliminaries} \label{Sect:Preliminaries}

Two topological dynamical systems $(X_1,f_1)$ and $(X_2,f_2)$ are said to be topologically conjugate if there exists a homeomorphism $\pi : X_1 \to X_2$ such that $f_2 \circ \pi = \pi \circ f_1$.



Let us now recall Bowen's definition of entropy for non-compact sets. Suppose $(X,f)$ is a topological dynamical system. For a finite open cover $\mathcal{U}$ of $X$ and $E \subset X$, we write $E \prec \mathcal{U}$ if $E$ is contained in an element of $\mathcal{U}$. Also, let 
\begin{equation*}
 n_{\mathcal{U}}(E) = \left\{ \begin{array}{ll}
                               0, & \text{ if } E \nprec \mathcal{U}, \\
                               \sup\{ k : f^k(E) \prec \mathcal{U} \}, & \text{ otherwise}.
                              \end{array}
                      \right.
\end{equation*}
For $\lambda \in \R$, define a measure $m_{\mathcal{U},\lambda}$ on Borel sets $Y \subset X$ by
\begin{equation*}
 m_{\mathcal{U},\lambda}(Y) = \lim_{N \to \infty} \inf \biggl\{ \sum_{i} \exp( - \lambda n_{\mathcal{U}}(E_i) ) : \bigcup_i E_i \supset Y, \, \inf_i n_{\mathcal{U}}(E_i) \geq N \biggr\}.
\end{equation*}
Note that if $\lambda > \lambda'$, then $m_{\mathcal{U},\lambda}(Y) \leq m_{\mathcal{U},\lambda'}(Y)$, and there exists at most one $\lambda$ such that $m_{\mathcal{U},\lambda}(Y) \notin \{ 0, \infty\}$. We may therefore define
\begin{equation*}
 h_{\mathcal{U}}(f,Y) = \inf\{ \lambda : m_{\mathcal{U},\lambda}(Y) = 0\}.
\end{equation*}
Then let
\begin{equation*}
 h_X(f,Y) = \sup\{ h_{\mathcal{U}}(f,Y) : \mathcal{U} \text{ is a finite open cover of } X\}.
\end{equation*}
Bowen showed in the case that $X$ is compact metrizable that the usual topological entropy $h(f)$ equals $ h_X(f,X)$. 
We recall next other basic facts about Bowen's  entropy.


\begin{prop} \label{Prop:Basic}\cite[Proposition 2]{Bowen1973}
\begin{enumerate}
 \item If $Y \subset X_1$ and  $\pi$ is a topological conjugacy from  $(X_1,f_1)$ to  $(X_2,f_2)$, then $h_{X_1}(f_1,Y) = h_{X_2}(f_2,\pi(Y))$;
 \item $h_X(f,\cup_{i=1}^{\infty} Y_i) = \sup_i h_X(f,Y_i)$;
 \item $h_X(f^m,Y) = m h_X(f,Y)$ for $m >0$. 
\end{enumerate}
\end{prop}
Despite (3) above, note that $h_X(f^{-1},Y)$ and $h_X(f,Y)$ are not necessarily equal.

We now turn our attention to shifts of finite type. Let $\mathcal{A}$ be a finite set. Let $\Sigma(\mathcal{A}) = \mathcal{A}^{\Z}$, which we endow with the product topology inherited from the discrete topology on $\mathcal{A}$, and define a homeomorphism $\sigma : \Sigma(\mathcal{A}) \to \Sigma(\mathcal{A})$ by $\sigma(x)_n = x_{n+1}$. For $x$ in $\Sigma(\mathcal{A})$ and integers $i \leq j$, let $x[i,j] = x_i \dots x_j$. 
For a square matrix $A = (A(i,j))_{i,j \in \mathcal{A}}$ with entries in $\{0,1\}$ and no zero rows or columns, define
\begin{equation*}
X_A = \{ x \in \Sigma(\mathcal{A}) : \forall i \in \Z, A(x_i, x_{i+1}) =1 \}.
\end{equation*}
We say a dynamical system $(X,f)$ is a shift of finite type (SFT) if there exists an alphabet $\mathcal{A}$ and a matrix $A$ as above such that $(X,f)$ is topologically conjugate to $(X_A,\sigma_A)$. We may write $(X,\sigma_X)$ or simply $X$ to refer to an SFT $(X,\sigma|_{X})$. An SFT $X$ is \textit{mixing} if for any two non-empty open sets $U,V \subset X$, there exists $N \geq 0$ such that if $ n \geq N$, then $\sigma^{-n}(U) \cap V \neq \emptyset$. If $X = X_A$, then $X$ is mixing if and only if there exists $n$ such that each entry of $A^n$ is positive.


Suppose $X$ and $Y$ are SFTs. A continuous map $\pi : X \to Y$ is a \textit{factor map} if $\pi$ is surjective and $\sigma_Y \circ \pi = \pi \circ \sigma_X$. A factor map $\pi : X \to Y$ is called \textit{left-closing} if $\pi$ never collapses forward-asymptotic points, \textit{i.e.} if whenever $\pi(x) = \pi(x')$ and there exists $n \in \Z$ such that $x[n,\infty) = x'[n,\infty)$, then it follows that $x = x'$. In other words, $\pi$ is injective on stable sets.

A point $x$ in an SFT $X$ is said to be \textit{doubly transitive} if every block in $X$ appears infinitely often in both $x[0,\infty)$ and $x(-\infty,0]$. Suppose $\pi : X \to Y$ is a factor map between mixing SFTs. We say that $\pi$ is \textit{almost invertible} if every doubly-transitive point in $Y$ has exactly one pre-image. 

Suppose $(X,\sigma_X)$ and $(Y,\sigma_Y)$ are SFTs on alphabets $\mathcal{A}$ and $\mathcal{B}$, respectively. A factor map $\pi : X \to Y$ is called a \textit{one-block code} if there exists $\Phi : \mathcal{A} \to \mathcal{B}$ such that $\pi(x)_n = \Phi(x_n)$ for all $n$ in $\Z$. Suppose $\pi : X \to Y$ is a one-block code. Let $w = w_1\dots w_m$ be an $m$-block in $Y$ and $1 \leq i \leq m$. Define $d^*_{\pi}(w,i)$ to be the number of symbols $a$ in $\mathcal{A}$ such that there is a word $u = u_1 \dots u_m$ in $X$ such that $\Phi(u) = w$ and $u_i = a$. Let $d^*_{\pi} = \min \{ d_{\pi}^*(w,i) \}$. Then a \textit{magic word} is a word $w$ in $Y$ such that $d_{\pi}^*(w,i) = d_{\pi}^*$ for some $i$. By passing to conjugate SFTs, we may assume without loss of generality that $w$ is a magic symbol, \textit{i.e.} $|w| = 1$ (see \cite[Proposition 9.1.7]{LindMarcus}). If $w$ is a magic symbol, then we may assume without loss of generality that $w \in \mathcal{A} \cap \mathcal{B}$. We require the following fact (for reference, see \cite[Theorems 9.1.11 and 9.2.2]{LindMarcus}).
\begin{fact} \label{Fact:MagicSymbol}
Suppose $X$ and $Y$ are mixing SFTs. Then a one-block factor map $\pi : X \to Y$ is almost invertible if and only if $d_{\pi}^* = 1$.
\end{fact}


Next, we define the left ideal class $\J(X)$ of a mixing SFT $X$ (for a full presentation, see \cite{BMT,LindMarcus}).
Given a subset $E$ of a ring $\mathfrak{R}$,  let $<E>_{\mathfrak{R}}$ denote the ideal in $\mathfrak{R}$ generated by $E$. Two ideals $\mathfrak{a}$ and $\mathfrak{b}$ are said to be equivalent, denoted $\mathfrak{a} \sim \mathfrak{b}$, if there exist nonzero $s,t$ in $\mathfrak{R}$ such that $ s \mathfrak{a}=t \mathfrak{b} $. 
Equivalence classes under $\sim$ are called \textit{ideal classes}. The ideal class of $\mathfrak{a}$ is denoted by $[\mathfrak{a}]_{\mathfrak{R}}$.

Let $X = X_A$ be a mixing SFT. Then $A$ has a Perron eigenvalue $\lambda_A$ and left eigenvector $v_A = (v_A^1,\dots,v_A^n)$. We assume without loss of generality that the entries of $v_A$ are all positive and contained in $\Z[\lambda_A] \subseteq \Z[1/\lambda_A]$. Then the left ideal class of $X$, denoted $\J(X)$, is defined as the ideal class of $< \{v_A^1 \dots, v_A^n\} >_{\Z[1/\lambda_A]}$. (Given $A,B$ there is an algorithm which decides whether $\J(X_A)=\J(X_B)$. This follows from the work of Kim and Roush showing decidability of shift equivalence \cite{KimRoush}, with a little more argument.) 


\begin{thm} \cite[Theorem 7.1]{BMT} \label{Thm:BMT} 
Let $X_A$ and $X_B$ be mixing SFTs with $\lambda = \lambda_A = \lambda_B$. Then the following are equivalent:
\begin{enumerate}
\item there exists a mixing SFT $X_C$ with $\lambda_C = \lambda$ and almost invertible, left-closing factor maps $\pi_A : X_C \to X_A$ and $\pi_B : X_C \to X_B$;
\item $\J(X_A) = \J(X_B)$ in $\Z[1/\lambda]$.
\end{enumerate}
\end{thm}


%

\begin{rmk} We notice that Bowen's conjecture is true if (and only if) every mixing SFT is entropy-conjugate to its inverse. To see this, suppose  $X$ and $Y$ are mixing SFTs with $h(X) = h(Y) = \log \lambda$. By the Adler-Marcus Theorem \cite{AdlerMarcus}, there exists a mixing SFT $Z$ and almost-invertible factor maps 
$\pi_1 : (Z,\sigma_Z) \to (X,\sigma_X)$ and $\pi_2 : (Z,\sigma_Z) \to (Y,\sigma_Y^{-1})$ such that $\pi_1$ is 
left-closing and $\pi_2$ is right-closing. Then as a map from $(Z,\sigma_Z^{-1}) \to (Y,\sigma_Y)$, $\pi_2$ is left-closing.  So, as in our proof of Theorem \ref{Thm:MainResult}, there are sets $E_1,E_1'$ satisfying the conditions for entropy-conjugacy with $\pi_1: Z\setminus E_1  \to X\setminus E_1'$ a homeomorphism, and similarly there are sets $E_2,E_2'$ for $\pi_2$. Now suppose there are likewise sets $E_3,E_3'$ and a topological conjugacy $\phi: (Z\setminus E_3,\sigma) \to (Z \setminus E_3',\sigma^{-1})$ giving an entropy-conjugacy. Because the maps $\pi_1,\pi_2$ are continuous factor 
maps of compact systems 
\cite{OprochaZhang}
(or by an elementary exercise for our maps), we have $h_X(\pi_1E_3,\sigma_X)\leq h_Z(E_3,\sigma_Z)$ and  $h_X(\pi_2E_3',\sigma_Y)\leq h_Z(E_3',\sigma_Z^{-1})$. 
Consequently there is an entropy-conjugacy of $(X,\sigma_X)$ and $(Y,\sigma_y) $ given by \[
\pi_2\phi\pi_1^{-1}: (X\setminus (E'_1\cup \pi_1E_3),\sigma)\to 
(Y\setminus (E'_2 \cup \pi_2E_3'),\sigma)\ . 
\]

By the above argument, any pair of ideal classes can occur as the left and right ideal class of a mixing SFT.
\end{rmk}

\begin{rmk}
 Suppose $S$ and $T$ are mixing SFTs of equal entropy. By the Adler-Marcus Theorem (with the left-resolving map down to $S$), there exists $S'$, the complement of a Bowen-negligible set in $S$, and a continuous embedding $f$ from $S'$ into $T$. Switching roles, there exists $T'$, the complement of a Bowen-negligible set in $T$, and a continuous embedding $g$ from $T'$ into $S$. However, this does not prove there is a topological entropy-conjugacy of $S$ and $T$.

 One can deduce from these mutual continuous embeddings that any two mixing SFTs of equal entropy are Borel conjugate after discarding the periodic points (a countable set). But the conjugating map thus obtained is definitely not continuous and therefore does not answer Hochman's question or decide Bowen's conjecture.
\end{rmk}

\section{Proof of Theorem \ref{Thm:MainResult}} \label{Sect:Proof}

The basic idea is that by Theorem \ref{Thm:BMT}, if $h(X) = h(Y)$ and $\J(X) = \J(Y)$, then there exists a common almost invertible, left-closing extension $Z$ of $X$ and $Y$. We exclude certain subsets of $X,Y$, and $Z$ that have small entropy, and on the remainder the almost invertibility and the left-closing property imply that these three systems are topologically conjugate.

We now set some notation and state two easy facts.
For a subset $Y$ of $\Sigma(\mathcal{A})$, let
\begin{equation*}
 W_n(Y) = \{ w \in \mathcal{A}^n : \exists y \in Y,  \, y[0,n-1] = w\}.
\end{equation*}


\begin{lem} \label{Lemma:Pesin}
 Let $X$ be an SFT, and let $Y \subset X$. Then 
\begin{equation*}
 h_X(\sigma_X, Y) \leq \lim_n \frac{1}{n} \log | W_n(Y) |.
\end{equation*}
\end{lem}


\begin{lem} \label{Lemma:BnEstimate}
 Let $X$ be an SFT with entropy $h(X) = \log \beta$. Then there exists a polynomial $p(x)$ such that
\begin{equation*}
 |W_n(X)| \leq p(n) \beta^n.
\end{equation*}
\end{lem}

\vspace{2mm}

\begin{PfofThmMain}{}
Suppose $X$ and $Y$ are mixing SFTs such that $h(X) = h(Y) = \log \lambda$ and $\J(X) = \J(Y)$. Then by Theorem \ref{Thm:BMT},
there exists a mixing SFT $Z$ and almost invertible, left-closing factor maps $\pi_1 : Z \to X$ and $\pi_2 : Z \to Y$. We assume without loss of generality that $\pi_1$ is one-block with a magic symbol $a$ and $\pi_2$ is one-block with a magic symbol $b$  (using Fact \ref{Fact:MagicSymbol} and recoding). 


For $n_0$ in $\Z$, define
\begin{equation*}
E_a(n_0) = \{ z \in Z :  \# \{ n \geq n_0 : z_n = a \} = 0 \}. 
\end{equation*}
Let $Z(a)$ be the SFT obtained by forbidding the symbol $a$ from $Z$. Let $h(Z(a)) = \log \beta$, and note that $h(Z(a)) < h(Z)$ by 
\cite[Corollary 4.4.9]{LindMarcus}. By Lemmas  \ref{Lemma:Pesin} and \ref{Lemma:BnEstimate}, there exists a polynomial $p(x)$ such that
\begin{align} \label{Eqn:MainEstimate}
\begin{split}
h_Z(\sigma_Z, E_a(n_0)) & \leq  \lim_n \frac{1}{n} \log( | W_n(E_a(n_0)) | ) \\ & \leq \lim_n \frac{1}{n} \log\biggl(  p(n) \lambda^{\max(0,n_0)} \beta^n \biggr) = h(Z(a)).
\end{split}
\end{align}
Let $E_a = \cup_{n_0 \in \Z} E_a(n_0)$. Then by Proposition \ref{Prop:Basic} and (\ref{Eqn:MainEstimate}),
\begin{equation}\label{Eqn:Ea}
h_Z(\sigma_Z, E_a) = \sup_{n_0} \biggl\{ h_Z(\sigma_Z, E_a(n_0)) \biggr\} \leq h(Z(a)) < h(Z).
\end{equation}
Define $E_b$ and $Z(b)$ analogously. By the analogous argument, we have that
\begin{equation} \label{Eqn:Eb}
 h_Z(\sigma_Z, E_b) \leq h(Z(b))< h(Z).
\end{equation}

Now let $E= E_a \cup E_b$. By Proposition \ref{Prop:Basic}, (\ref{Eqn:Ea}), and (\ref{Eqn:Eb}), we have that
\begin{equation*}
h_Z(\sigma_Z, E) = \max( h_Z(\sigma_Z, E_a), h_Z(\sigma_Z, E_b)) < h(Z).
\end{equation*} Furthermore, define $B_X = \pi_1(E)$ and $B_Y = \pi_2(E)$. Since entropy of a continuous map defined on a metrizable compact space cannot increase under a factor map,
\begin{equation} \label{Eqn:BX}
 h_X(\sigma_X, B_X) \leq h_Z(\sigma_Z, E) < h(Z)= h(X),
\end{equation}
and similarly, 
\begin{equation} \label{Eqn:BY}
 h_Y(\sigma_Y, B_Y) \leq h_Z(\sigma_Z, E) < h(Z) = h(Y)
\end{equation}
Set $X' = X \setminus B_X$, $Y' = Y \setminus B_Y$, and $Z' = Z \setminus E$. Note that $\sigma_X(X') \subset X'$, $\sigma_Y(Y') \subset Y'$, and $\sigma_Z(Z') \subset Z'$.

Let us prove that $\pi_1|_{Z'}$ is a homeomorphism between $Z'$ and $X'$. Since $\pi_1$ is continuous, $\pi|_{Z'}$ is continuous. Let $x$ be in $X'$. Then $x \notin \pi_1(E)$. Since $\pi_1$ is onto, we conclude that $x \in \pi_1(Z')$, and therefore $X' \subset \pi_1(Z')$. Now let $x$ be in $\pi_1(Z')$, \textit{i.e.} suppose there exists $z$ in $Z'$ such that $\pi_1(z) = x$. Since $z \notin E$, there exist $n_k \to \infty$ such that $z_{n_k} = a$. Since $a$ is a magic letter, this fact implies that $z[n_0,\infty)$ is uniquely determined by $x$. Furthermore, since $\pi_1$ is left-closing, we see that $z$ is uniquely determined by $x$. Hence $\pi_1^{-1}(x) = \{z\}$ and $x\notin\pi_1(E)$. From this argument, we draw two conclusions:
\begin{enumerate}
 \item $X' = \pi_1(Z')$, and
 \item $\pi_1|_{Z'}$ is injective.
\end{enumerate}
Thus, we have shown that $\pi_1|_{Z'}$ is a continuous bijection from $Z'$ onto $X'$. Furthermore, we claim that $\pi_1|_{Z'}^{-1}$ is continuous. To see this, note that $\pi_1$ is left-closing if and only if there exists $K$ such that whenever $\pi_1(z[-K,K]) = \pi_1(z'[-K,K])$ and $z[0,K] = z'[0,K]$, then $z_1 = z'_1$ \cite[Proposition 8.1.9]{LindMarcus}. Now let $x$ be in $X'$. It suffices to show that there exists an $N = N(x)$ such that $x[-N,N]$ determines $\pi_1^{-1}(x)_0$. Since $x$ is in $X'=\pi_1(Z')$, there exists $n_k \to \infty$ such that $x_{n_k} = a$. Choose $n_k$ such that $n_k - n_1 \geq K$. Since $a$ is a magic symbol, $\pi_1^{-1}(x)[n_1,n_k]$ is determined by $x[n_1,n_k]$. Furthermore, since $\pi_1$ is left-closing and $n_k - n_1 \geq K$, we see that $\pi_1^{-1}(x)[0,n_k]$ is determined by $x[-K,n_k]$. Thus, $\pi_1|_{Z'}^{-1}$ is continuous.

By the analogous argument, we have that $\pi_2|_{Z'}$ is a homeomorphism from $Z'$ onto $Y'$. Now let $\pi = \pi_2 \circ \pi_1^{-1}|_{X'}$, which is then a homeomorphism from $X'$ onto $Y'$. Combining this fact with (\ref{Eqn:BX}) and (\ref{Eqn:BY}), we obtain that $\pi$ is an entropy-conjugacy between $X$ and $Y$.
\end{PfofThmMain}

\section*{Acknowledgements}

MB was supported by the Danish National Research Foundation through
the Centre for Symmetry and Deformation (DNRF92). 
JB was partially supported by ANR grant DynNonHyp (BLAN08-2$\underline{\;}$313375).
KM gratefully acknowledges the support of NSF-DMS grant number 10-45153.

\bibliographystyle{plain}
\bibliography{BowenEntConj_Refs}





\end{document}